\documentclass[a4paper,11pt]{amsart}
\usepackage[margin=25mm]{geometry} 
\usepackage{amsmath, amssymb, amsthm}
\usepackage{booktabs, multirow, url, cite}

\makeatletter
\newcommand{\udot}{\mathpalette\udot@\relax}
\newcommand{\udot@}[2]{%
    \begingroup
    \sbox\z@{$#1{:}$}%
    \sbox\tw@{$#1{.}$}%
    \raisebox{\dimexpr\ht\z@-\ht\tw@}{$\m@th#1.$}%
    \endgroup
}

\def\mmgroup{\texttt{mmgroup}}
\def\etal{et~al.\ }

\newtheorem{lemma}{Lemma}[section]
\newtheorem{remark}[lemma]{Remark}
\newtheorem{prop}[lemma]{Proposition}

\newcommand{\CO}{\textup{Co}_1}
\newcommand{\QStr}{2^{1+24}}

\newcommand{\psl}[2]{\textup{PSL}_{#1} ( #2 )}

\newcommand{\sym}[1]{\textup{S}_{#1}}

\newcommand{\grp}{2^{5+10+20} \udot ( \sym{3} \times \psl{5}{2} )}
\newcommand{\grpA}{\QStr \udot \CO}
\newcommand{\grpB}{2^{2+11+22} \udot ( {\rm M}_{24} \times \sym{3} )}
\newcommand{\grpC}{2^{3+6+12+18} \udot ( \psl{3}{2} \times 3.\sym{6} )}
\newcommand{\grpZ}{2^{10+16} \udot {\rm O}^+_{10} (2)}

\author{Anthony Pisani}
\address{School of Mathematics, Monash University, Clayton VIC 3800, Australia}
\email{anthony.pisani@monash.edu}
\date{\today}
\title{Computing the Character Table of a 2-local Maximal Subgroup of the Monster}

\begin{document}
    \begin{abstract}
        We employ the recently developed hybrid and \mmgroup{}
        computational models for groups to calculate the character table
        of $N({\rm 2B}^5) \cong \grp$, a maximal subgroup of
        the Monster sporadic simple group. This completes the list
        of the character tables of maximal subgroups of the Monster.
        Our approach illustrates how the aforementioned computational
        models can be used to calculate relatively straightforwardly
        in the Monster.
    \end{abstract}
    
    \maketitle
    
    \section{\label{sec:intro}Introduction}
    
    \noindent The maximal subgroups of the Monster group, the largest of
    the 26 sporadic simple groups, form the focus of decades of research
    by Holmes, Norton, Wilson, and others. Wilson \cite{survey} observed
    in 2017 that their classification was the topic of more than 15 papers,
    with a small number of cases still outstanding; these were
    finally resolved in a 2023 paper of Dietrich \etal\cite{dlp} using
    Seysen's innovative Python package \mmgroup{}
    \cite{serpent,fast_monster,mmgroup}. We refer to \cite{dlp}
    for a more comprehensive discussion of the classification and
    its history. Meanwhile, calculations of their character tables and
    class fusions to the Monster span a series of works
    \cite{atlas,dlp,four_fus,n3b,verif}, most recently Hulpke's \cite{arith}
    computation of the character table for the class of subgroups $\grpZ$
    using a new hybrid representation developed with Dietrich \cite{dhulp}.
    (Atlas \cite{atlas} notation, e.g. $\grpZ$, will be used without
    explanation throughout.)
    Such data have been carefully collected in the GAP Character Table
    Library \cite{gap_ctbllib}, which as of its most recent version (1.3.9)
    covers all classes of maximal subgroups of the Monster except those of
    shape $\grp$, with attribution information provided via the library's
    change log and \texttt{InfoText} attributes of character tables. \\
    
    The development of suitable computer representations has proved
    a critical bottleneck in all investigations.
    Although theoretical approaches revealed great deal about the Monster
    early on, including its maximal $p$-local subgroups\footnote{
        :~That is, $p$-local subgroups not contained in a larger
        $p$-local subgroup. This is only a necessary condition
        for maximality, though many such subgroups of the Monster are
        in fact maximal.
    } through work of Wilson \cite{odd_local} (building on Norton) and
    Meierfranenfeld and Shpectorov \cite{2local1,2local2}, the limits of
    existing techniques were reached by the turn of the millenium.
    On the other hand, the Monster's lack of ``small representations'' has
    rendered it --- and, similarly, some of its larger subgroups ---
    impervious to the standard methods of computational group theory.
    Considerable progress has been made over time: unpublished software
    of Linton \etal\cite{mono} and later Holmes and Wilson \cite{monp}
    has been used to great effect for calculating in the Monster, while
    a number of its maximal subgroups have been implemented and studied via
    specialised constructions (e.g. \cite{n3b}) or simply
    significant computing power. However, a straightforward, unified approach
    that allows groups to be easily constructed, shared, and analysed
    \textsl{as subgroups embedded in the Monster} has until recently remained
    beyond reach. \\
    
    As the first open-source tool capable of rapid computation
    in the Monster group, Seysen's \mmgroup{} \cite{mmgroup} represents
    a ground-breaking advance in this respect. An early application
    by Dietrich \etal\cite{dlp} completes the classification of
    maximal subgroups of the Monster using reproducible calculations
    and constructions, provided online alongside the paper. Later work
    \cite{dlpp} offers instances of almost all maximal subgroups
    in a similarly well-documented format. \\
    
    Dietrich and Hulpke's hybrid representation \cite{dhulp,arith,hybrep} is
    likewise an important development. Despite \mmgroup's achievements,
    it currently lacks much of the functionality present
    in more standard group-theoretic tools such as GAP \cite{gap}.
    A hybrid representation of a finite group augments a PC (polycyclic) presentation
    \cite[\S8]{hcgt} for the group's solvable radical with extra generators
    and relators, as well as a confluent rewriting system and
    permutation represention for the factor group these additional elements
    generate. This allows effective encoding of groups in which the radical
    is large (such as many $p$-local subgroups of the Monster). Thanks
    to the wealth of prior work on the more general Solvable Radical Paradigm
    \cite[\S10]{hcgt}, a variety of algorithms for computing
    with the hybrid representation are already available and implemented
    in GAP. See \cite{arith} and the references therein
    for further details, including an application to calculate
    the character table of $\grpZ$. \\
    
    The aim of this paper is to illustrate an intuitive new approach
    to working with maximal subgroups of the Monster by calculating
    the character table of the subgroup $\grp$. We directly construct
    an instance of this subgroup within a copy of the Monster --- provided
    in \cite{dlpp} --- and derive a presentation therein to pass to GAP,
    as described in \S\ref{sec:calc_g}. A hybrid group constructed
    from this presentation is then employed in all necessary calculations
    to obtain the character table. Our task can mostly be
    accomplished with standard GAP commands in the span of a few days,
    with code provided online at \cite{pisgit} for easy reproducibility;
    moreover, easy transfer between the \mmgroup{} and hybrid representations
    means results can be considered in the context of the Monster. \\
    
    The character table computed for $\grp$ is to be included
    in an upcoming release of the GAP Character Table Library. Its inclusion
    will simplify many arguments about the Monster, and in fact
    sporadic groups more generally, by allowing one to iterate
    over the character tables of all their maximal subgroups. \\
    
    \subsection*{\label{ack}Acknowledgments}
    
    The author is deeply grateful to Dr. Melissa Lee
    for many helpful suggestions in relation to the paper. The author
    would also like to thank Professor Thomas Breuer
    for advice on working with character tables in GAP and
    the GAP Character Table Library, Professors Heiko Dietrich and Alexander Hulpke
    for their support with the hybrid group representation in GAP,
    and Dr. Tomasz Popiel for originally suggesting the problem.
    This research was supported by an Australian Research Council Discovery
    Early Career Researcher Award (project number DE230100579).
    
    \section{\label{sec:calc_g}Initial Computations}
    
    Let $G$ denote --- here and in all that follows --- the maximal subgroup of
    the Monster with shape $\grp$ constructed using \mmgroup{} in \cite{dlpp}.
    As outlined in \S\ref{sec:intro}, the first task is to represent $G$
    as a hybrid group. Recall that such a group is given by extending
    a polycyclic (PC) presentation \cite[\S8]{hcgt} for the group's solvable radical ${\textbf O}_\infty (G)$
    with extra generators and relators, along with useful but redundant data
    describing the factor group $G/{\textbf O}_\infty (G)$ these additions generate.
    It will be useful to note that the further relators are divided
    into two sets \cite[\S1]{arith}. One describes the automorphisms of
    ${\textbf O}_\infty (G)$ that the extra generators induce by conjugation.
    The other consists of relators for $G/{\textbf O}_\infty (G)$: due to the fact
    that these will in general evaluate to a non-trivial element of
    ${\textbf O}_\infty (G)$ rather than the identity, however, they must be suffixed
    with elements $t \in {\textbf O}_\infty (G)$ (Hulpke's ``tails'' \cite{arith})
    to give relators for $G$. \\
    
    With this in mind, we seek solvable $N \, \lhd \, G$ in \mmgroup{} such that
    we can evaluate isomorphisms from $N$ to a polycyclic group, from $G/N$
    to a permutation group, and their inverses. This is sufficient to derive
    a PC presentation for $N ({\textbf O}_\infty (G)/N) = {\textbf O}_\infty (G)$ and
    a presentation for $G/{\textbf O}_\infty (G)$, which in turn allows
    us to compute the required conjugation automorphisms and
    relator corrections $t \in {\textbf O}_\infty (G)$. Further details are
    supplied in Remark~\ref{rmk:construction}. The following lemma
    simply verifies the resulting hybrid group is isomorphic to $G$.
    
    \begin{lemma}
        \label{lem:g}
        The hybrid group constructed by the file \texttt{2\_5.g}
        in the accompanying code \cite{pisgit} is isomorphic to $G$.
    \end{lemma}
    \begin{proof}
        Let $H$ denote the group constructed by \texttt{2\_5.g}, and
        $\mathcal{S}$ its solvable radical. The script \texttt{calc.g}
        first confirms that
        $|H| = 2^{46} \cdot 3^3 \cdot 5 \cdot 7 \cdot 31 = |G|$, before
        writing relators for $H/\mathcal{S}$, the automorphisms
        of $\mathcal{S}$ induced by generators of $H/\mathcal{S}$,
        the elements $t \in \mathcal{S}$ required to lift the relators
        for $H/\mathcal{S}$ to $H$, and PC relators for $\mathcal{S}$
        to a file \texttt{pres.json}. As noted above, the data
        in \texttt{pres.json} constitute a presentation for $H$.
        We then use \texttt{2\_5.py} to verify that certain
        \mmgroup{} elements $L_G \subset G$ both satisfy
        this presentation and generate $G$. It follows
        by Von Dyck's theorem \cite[Theorem~2.53]{hcgt} that $G$ is
        an epimorphic image of $H$, which since $|G| = |H|$ can
        only be $H$ itself.
    \end{proof}
    
    \begin{remark}
        \label{rmk:construction}
        It is worth noting how the hybrid group was constructed originally
        from the \mmgroup{} representation. A normal elementary abelian subgroup
        $N = \langle g^G \rangle$ of $G$ with order $2^{15}$ was
        obtained first, where suitable $g \in G$ was found by random search.
        Enumeration of $N$ as words in a set of independent generators
        then gave an isomorphism from $N$ to the vector space
        $\mathbb{F}_2^{15}$, allowing the conjugation action of $G$
        on $N$ to be represented as a group of $15 \times 15$ matrices
        over $\mathbb{F}_2$. It turns out that the resulting matrix group
        induces an orbit on $1488$ vectors in $\mathbb{F}_2^{15}$, and that
        the permutation group obtained from this action has order $|G|/2^{15}$.
        We thus have a faithful permutation representation of $G/N$. A base and
        strong generating set could then be constructed to allowing mapping of
        permutations to pre-images in $G$, which in combination
        with the isomorphism on $N$ provided the necessary infrastructure
        to extend a hybrid representation of $G/N$ --- easily computed in GAP
        --- to $G$.
    \end{remark}
    
    Once this has been done, conjugacy classes and other elementary data about $G$
    required for the header of its character table can be computed with GAP's in-built
    facilities, as illustrated in the accompanying code file \texttt{calc.g}
    \cite{pisgit}. No special comment is necessary, except that most power maps can
    be deduced from a small subset via the relation $(x^i)^j = x^{ij}$.
    
    \section{\label{sec:chartab}The Characters of $G$}
    
    Once the appropriate preliminary data --- the header of the character table ---
    has been computed, the task turns to determining the list of
    irreducible characters of $G$. We follow Unger \cite{unger},
    who describes the default algorithm for computing character tables used by Magma,
    in applying a variant of Brauer's Theorem on Induced Characters:
    
    \begin{prop}
        \label{prop:brauer}
        Let $H$ be a finite group with conjugacy class representatives
        $h_1, h_2, \ldots h_n$. For each $h_i$, fix representatives
        $P_{i, p}$ of those classes of non-cyclic Sylow subgroups
        of $C_H(h_i)$ with order coprime to $|h_i|$. Then
        the irreducible characters of $H$ are linear combinations
        with integer coefficients of those induced from subgroups
        in $\mathcal{S} = \{
            \langle h_i \rangle,
            P_{i, p} \times \langle h_i \rangle
        \}$.
    \end{prop}
    \begin{proof}
        Serre \cite[\S10.4, Theorem 19]{serre} proves that
        the irreducible characters of $H$ are such combinations
        of the characters induced from subgroups in $\mathcal{S}_2 = \{
            \langle h_i \rangle,
            Q_{i, p} \times \langle h_i \rangle
        \}$, where the subgroups $Q_{i, p}$ are representatives of
        all classes of Sylow subgroups of $C_H(h_i)$ with order coprime
        to $|h_i|$, cyclic or not. (The theorem as stated by Serre allow
        for a much larger class of subgroups, but a careful reading of
        the proof reveals only the aforementioned subgroups are used.)
        But $\mathcal{S} = \mathcal{S}_2$: if $Q_{i, p}$ is cyclic, then
        $Q_{i, p} \times h_i$ is a direct product of cyclic groups
        with coprime orders, and thus itself cyclic
        by the Chinese Remainder Theorem.
    \end{proof}
    
    \begin{remark}
        Since $|G| = 2^{46} \cdot 3^3 \cdot 5 \cdot 7 \cdot 31$,
        the sets $P_{i, p}$ for $G$ are 2- or 3-groups; moreover,
        the 3-groups are of order at most $27$. It turns out that there
        are a total of $19$ 2-groups and $27$ 3-groups to consider,
        the latter all elementary abelian.
    \end{remark}
    
    Now, constructing a set $\mathcal{S}$
    as in Proposition~\ref{prop:brauer} is computationally straightforward.
    A little more thought is required to determine, given a list of
    the characters of a finite group $H$ induced from subgroups
    in $\mathcal{S}$, which linear combinations with integral coefficients
    give rise to irreducible characters of $H$. The well-known result that
    irreducible characters (and their negatives) are
    the only such combinations that have norm 1 with respect
    to the inner product $\langle \chi, \psi \rangle =
        \left| H \right|^{-1} \sum_i
            \left| h_i^G \right| \chi (h_i) \bar{\psi} (h_i)$
    provides a useful test to this end. Indeed, GAP's \cite{gap}
    implementation of the LLL lattice reduction algorithm \cite[\S17]{galb},
    which seeks integral linear combinations of vectors minimising
    a given norm, can be used with this inner product to search
    for irreducible characters initially. This suffices to yield
    the full character table of $H = G$, though Hulpke \cite[\S4]{arith} warns
    that its success is not guaranteed in general. \\
    
    The only remaining --- and most time-consuming ---
    step in implementing Proposition~\ref{prop:brauer} is the computation
    of characters induced from subgroups in $\mathcal{S}$. Observe that
    $\langle h_i \rangle = 1 \times \langle h_i \rangle$, so
    all such subgroups can be written in the form
    $P \times \langle x \rangle$ for a $p$-group $P$ and
    element $x$ of order coprime to $|P|$. The following lemma
    gives a formula for the characters of $H$ induced
    from irreducible characters of
    $P \times \langle x \rangle \in \mathcal{S}$.
    
    \begin{lemma}
        \label{lem:ind}
        Let $H$ be any finite group, $x \in H$ be any element of $H$,
        and $P \le C_H (x)$ be a group such that $|P|$ is coprime
        to $n = |x|$. Write $e$ for the exponent of $P$.
        Then there is a function $M : \, \mathbb{Z} \to \mathbb{Z}$
        such that $M(i) \equiv i \textrm{ mod } n$ and
        $M(i) \equiv 1 \textrm{ mod } e$ for all $i \in \mathbb{Z}$.
        Moreover, the characters of $H$ induced from
        irreducible characters of $K = P \times \langle x \rangle$ are
        given by
        
        \[ \chi (g) = \frac{|C_H(g)|}{n} \sum_{i=0}^{n-1} u^i \sum_{
                \substack{y \in \mathcal{C} \\
                    (xy)^{M(i)} \in g^G}
            } \frac{\chi(y)}{| C_P (y) |}, \qquad g \in H, \]
        
        where $\mathcal{C}$ is a set of conjugacy class representatives
        for $P$, $\chi$ is an irreducible character of $P$, and
        $u \in \mathbb{C}$ is any $n^{\textrm{th}}$ root of unity.
    \end{lemma}
    
    \begin{proof}
        Since $e$ divides $|P|$, which is coprime to $n$,
        the first claim follows from the Chinese Remainder Theorem. \\
        
        To prove the remainder of the lemma, note that
        the direct product $K = P \times \langle x \rangle$
        has irreducible characters given by $\chi_0 (x^i y) =
            \chi_1 (x^i) \chi_2 (y)$ for all $x^i \in \langle x \rangle$
        and $y \in P$, where $\chi_1$ and $\chi_2$ are
        irreducible characters of $\langle x \rangle$ and $P$, respectively.
        The irreducible characters of the cyclic group $\langle x \rangle$,
        furthermore, are of the form $\chi_1 (x^i) = u^i$ for $u$
        an $n^{\textrm{th}}$ root of unity. Hence, choosing $u$ and
        $\chi = \chi_2$ corresponding to any particular irreducible character
        $\chi_0$ of $K$,
        
        \begin{align*}
            {\rm ind}^H_K \chi_0 (g) & = \frac{1}{|K|} \sum_{
                \substack{h \in H \\ g^h \in K}
            } \chi_0 (g^h) \\
            & = \frac{1}{|P \times \langle x \rangle|} \sum_{
                k \in K \cap g^G
            } \chi_0 (k) |C_H(k)| \\
            & = \frac{|C_H(g)|}{|P| n} \sum_{i=0}^{n-1} \sum_{
                \substack{y_0 \in P \\ x^i y_0 \in g^G}
            } \chi_0 (x^i y_0) \\
            & = \frac{|C_H(g)|}{|P| n} \sum_{i=0}^{n-1} \sum_{
                y \in \mathcal{C}
            } \frac{1}{| C_P (y) |} \sum_{
                \substack{k \in P \\ \, x^i y^k \in g^G}
            } \chi_0 (x^i y^k).
        \end{align*}

        Now, elements $y, k \in P \le C_H(x)$ satisfy
        $x^i y^k = (x^k)^i y^k = (x^i y)^k = (x^{M(i)} y^{M(i)})^k = ((xy)^{M(i)})^k$
        by the definition of $M(i)$. Thus
        
        \begin{align*}
            {\rm ind}^G_K \chi_0 (g) & =
            \frac{|C_H(g)|}{|P| n} \sum_{i=0}^{n-1} \sum_{
                y \in \mathcal{C}
            } \frac{1}{| C_P (y) |} \sum_{
                \substack{k \in P \\ ((xy)^{M(i)})^k \in g^G}
            } u^i \chi(y^k) \\
            & = \frac{|C_H(g)|}{|P| n} \sum_{i=0}^{n-1} u^i \sum_{
                y \in \mathcal{C}
            } \frac{1}{| C_P (y) |} \sum_{
                \substack{k \in P \\ (xy)^{M(i)} \in g^G}
            } \chi(y) \\
            & = \frac{|C_H(g)|}{|P| n} \sum_{i=0}^{n-1} u^i \sum_{
                \substack{y \in \mathcal{C} \\
                    (xy)^{M(i)} \in g^G}
            } \frac{|P|}{| C_P (y) |} \chi(y) \\
            & = \frac{|C_H(g)|}{n} \sum_{i=0}^{n-1} u^i \sum_{
                \substack{y \in \mathcal{C} \\
                    (xy)^{M(i)} \in g^G}
            } \frac{\chi(y)}{| C_P (y) |}.
        \end{align*}
    \end{proof}
    
    The conjugacy classes of $G$ have already been computed,
    along with the corresponding orders, centraliser sizes, and power maps.
    It is therefore possible to apply Lemma~\ref{lem:ind}
    to $P \times \langle h_i \rangle < G$ once we know
    conjugacy class representatives for $P$, the centraliser order $|C_P (y)|$
    and $G$-class of $h_i y$ for each representative $y$, and
    the irreducible characters of $P$. As illustrated in \texttt{calc.g},
    standard GAP commands can in \textsl{most} cases compute the relevant data
    quickly and efficiently; the supersolvability of $p$-groups ensures that
    extremely efficient algorithms can be used, while no computation is
    even necessary for $P = 1$. \\
    
    \subsection{\label{sec:sylow}Sylow 2-Subgroups of $G$}
    
    The sole case in which the computations for Lemma~\ref{lem:ind}
    present serious difficulty is when $P$ is a Sylow 2-subgroup of
    $C_G(1) = G$. The $|P| = 2^{46}$ elements --- which also make $P$
    a Sylow 2-subgroup of the Monster, incidentally --- and especially
    the 26758 conjugacy classes of $P$ are so numerous as to make
    considerable demands on time and space when calculating.
    These are a mere inconvenience for the most part, but present
    a serious impediment to the enumeration of irreducible characters
    of $P$. A more considered approach must therefore be taken. \\
    
    Recall that the irreducible characters of a finite group $H$
    form a square array of algebraic integers, with dimension equal
    to the number of conjugacy classes of $H$. Conlon \cite{conlon}
    describes an efficient and now widely used algorithm to obtain
    this array for $H$ supersolvable, consisting of the construction
    of a list of pairs $(C, K)$ such that $K \lhd \, C \le H$, the $C/K$
    are central cyclic subgroups of the $C_H(K)/K$, and
    each irreducible character of $H$ is induced
    from an irreducible character of some $C$ with kernel $K$.
    Experimentation with the GAP \cite{gap} function \texttt{IrrConlon},
    which is a (slightly modified) implementation of Conlon's algorithm,
    shows that the vast majority of
    time and space are spent on inducing characters from small subgroups
    and storing what eventually amount to $26758^2 \approx 7.2 \times 10^8$
    character values, respectively. The following modifications, focused
    on these bottlenecks, are thus made to help alleviate the burden of
    computing the irreducible characters of $P \in {\rm Syl}_2 (G)$. \\
    
    First, we split \texttt{IrrConlon} into two functions: one which
    constructs the list of pairs $(C, K)$, and another which evaluates
    the characters induced from a single pair. Multiple GAP instances run
    in parallel, each clones of a single workspace in which
    the relevant pairs have been computed and stored, can then share
    the work of inducing characters. \\
    
    Second, we avoid direct computation of character values where
    possible via a variant of a result by Th\"ummel:
    
    \begin{prop}[Adapted from {Th\"ummel \cite{thummel}}, Lemma~2]
        \label{lem:thum}
        Let $C \le H$ be finite groups. Suppose that $C$ is
        generated by $K \lhd C$ and $g \in C$, where $C/K$ is
        cyclic. Furthermore, suppose $N \unlhd K$ is a normal subgroup
        of $H$ such that the coset $gN$ is central in $H/N$.
        Then, for all $x \in H$, any character $\chi$ induced
        from a character of $C$ with kernel $K$ satisfies 
        \begin{enumerate}
            \item $\chi(xn) = \chi(x)$ for every $n \in N$;
            \item $\chi(xg) = \omega \chi(x)$, where $\omega$
                is a $|C/K|^{\rm th}$ root of unity not
                dependent on $x$ (so that, in particular,
                $\omega = -1$ if $|C/K| = 2$); and consequently
            \item $\chi(x) = 0$ whenever $xg$ is conjugate to $xn$
                for some $n \in N$.
        \end{enumerate}
    \end{prop}
    
    Finally, rather than being retained as lists in RAM, characters
    are written straight to disk with a very basic run-length encoding.
    Doing so reduces the pressure of memory constraints, whilst also
    serving to unify output from multiple GAP processes. Note that
    the tendency of characters to take the same value across several classes,
    indicated by Lemma~\ref{lem:thum}, means run-length encoding yields
    quite large compression ratios (exceeding 16 for $P$). \\
    
    Using \texttt{IrrConlon} with the above modifications,
    25540 irreducible characters of $G$ were computed in the span of 3 weeks
    on consumer grade hardware (a Dell Latitude with 32GB RAM and a 4-core
    Intel i5 processor). The resulting list occupies 140MB of disk space.
    Calculations were halted once all 718 irreducible characters of $G$ could
    successfully be derived.
    
    \begin{remark}
        Note that the Sylow 2-subgroups of $G$ are also Sylow 2-subgroups
        of 2-local maximal subgroups of the Monster with shapes
        $\grpA$, $\grpB$, $\grpC$, and $\grpZ$, as well as of the Monster
        itself. The list of computed irreducible characters of this 2-group
        is available from the author upon request.
    \end{remark}
    
    \begin{small}
    
    \end{small}
\end{document}